\documentclass[12pt,a4paper]{article}
\usepackage{amsmath,amssymb,amsfonts,amsthm}
\usepackage{float,graphicx,color}
\usepackage[all]{xy}
\newcommand{\rmd}{\mathrm{d}}

\newcommand{\rmD}{\mathrm{D}}

\newcommand{\rmL}{\mathrm{L}}

\newcommand{\sch}{\textsc{h}}

\newcommand{\bbG}{\mathbb{G}}

\newcommand{\bbN}{\mathbb{N}}

\newcommand{\bbR}{\mathbb{R}}
\newcommand{\bbS}{\mathbb{S}}

\newcommand{\bbU}{\mathbb{U}}

\newcommand{\frg}{\mathfrak{g}}

\newcommand{\calC}{\mathcal{C}}

\newcommand{\calF}{\mathcal{F}}
\newcommand{\calG}{\mathcal{G}}

\newcommand{\calL}{\mathcal{L}}

\newcommand{\calO}{\mathcal{O}}

\newcommand{\calS}{\mathcal{S}}
\newcommand{\calT}{\mathcal{T}}

\DeclareMathOperator{\real}{\mathfrak{Re}}
\DeclareMathOperator{\tr}{tr}
\DeclareMathOperator{\rtr}{\real \tr}

\DeclareMathOperator{\ad}{\mathrm{ad}}
\DeclareMathOperator{\Ad}{\mathrm{Ad}}

\DeclareMathOperator{\SU}{\bbS\bbU}

\newcommand{\cleq}{\preceq}

\newcommand{\subsimp}{\triangleleft}

\newcommand{\myone}{\mathbf{1}}
\newcommand{\myzero}{\mathbf{0}}

\newcommand{\ts}{\textstyle}

\newcommand{\rmins}[1]{\quad \textrm{#1} \quad}

\newcommand{\bs}{{\scriptscriptstyle \bullet}}

\newcommand{\Gact}{}          %{\, {\scriptstyle \odot}\, }
\newcommand{\gact}{}          %{\, {\scriptstyle \boxdot}\, }

\newcommand{\mapping}[4]{
\left\{
\begin{array}{rcl}
\displaystyle #1  &\to& #2 ,\\
\displaystyle #3  &\mapsto      & #4.
\end{array} \right.
}

\newtheorem{lemma}{Lemma}
\newtheorem{proposition}{Proposition}

\newtheorem{theorem}{Theorem}
\theoremstyle{definition}\newtheorem{definition}{Definition}
\theoremstyle{remark}\newtheorem{remark}{Remark}

\begin{document}
\title{A simplicial gauge theory}
\author{Snorre H. Christiansen, Tore G. Halvorsen}
\date{}
\maketitle

\begin{abstract}
We provide an action for gauge theories discretized on simplicial meshes, inspired by finite element methods. The action is discretely gauge invariant and we give a proof of consistency. A discrete Noether's theorem that can be applied to our setting, is also proved. 
\end{abstract}

\section{Introduction}

The Standard Model describes fundamental particles and interactions, except gravity, as a quantum field theory with gauge group $\bbU(1) \times \SU(2) \times \SU(3)$, see e.g. \cite{Wei05I}\cite{Wei05II}. Lattice Gauge Theory (LGT) \cite{Wil74}\cite{Kog83}\cite{Rot05}  is a computational approach with good agreement with experimental data. It has proved particularly useful for the $\SU(3)$ sector, concerning quarks and gluons, where perturbative methods fail.

In LGT, the underlying space-time is discretized by a cubical lattice. Scalar fields are assigned degrees of freedom on vertexes. Gauge fields are assigned degrees of freedom on edges, representing parallel transport between vertexes (Wilson lines). Curvature is recovered from holonomies around squares (plaquettes) in the lattice (Wilson loops). These quantities are combined to form a discrete action (Wilson action) that is gauge invariant, with respect to gauge transformations at vertexes. The construction of a discrete gauge invariant action is a fundamental ingredient in LGT, and is the only part of LGT we are concerned with in this paper.

There is interest in constructing discrete gauge invariant actions on other underlying geometries, in particular simplicial decompositions of space-time. This has already received considerable attention\footnote{We were blissfully  ignorant about these, when submitting the first draft.}, starting with \cite{ChrFriLee82}\cite{DroMor83}. For developments motivated by non-commutative differential geometry, see \cite{DimMul94}. The approach we propose in this paper results in an alternative prescription and is inspired by finite element methods, as they have been developed for Maxwell's equations, corresponding to the $\bbU(1)$ sector.

The most successful finite elements for Maxwell's equations are those introduced in \cite{Ned80}, generalizing \cite{RavTho77}. As remarked in \cite{Bos88}, lowest order Raviart-Thomas-N\'ed\'elec elements, correspond to Whitney forms \cite{Wei52}\cite{Whi57}. Based on this connection, a finite element exterior calculus has been developed \cite{Hip99}\cite{Hip02}\cite{Chr07NM}\cite{ArnFalWin06}. We refer to \cite{ArnFalWin10} for a recent review, relating the analysis developed with engineering problems in mind, to discrete Hodge theory, as in \cite{DodPat76}.

Here, we tap into this large body of work, continuing our investigations of Lie algebra valued Whitney forms \cite{ChrWin06} on the one side and LGT \cite{ChrHal11IMA}\cite{ChrHal09BIT} on the other. The point of view we develop, uniting these strands, is that LGT should define an action on the space of Lie algebra valued Whitney forms, which is close to the restriction of the Yang-Mills action, yet invariant under some discrete gauge transformations. 

The Yang-Mills action is of course gauge invariant, but the space of Lie algebra valued Whitney forms is not, resulting in a violation of local charge conservation. The analogue for electromagnetics would be that the discretization would violate electric charge conservation.  A discrete gauge invariance on the other hand, should give local charge conservation by discrete Noether's theorems, mimicking the continuum theory. A key ingredient in any convergence analysis is to estimate the error between the restriction of the Yang-Mills action and the new one. In the context of finite element methods, such \emph{consistency} errors are analysed in the framework of variational crimes \cite{Str72AP}, see \cite{Cia91} \S 26 -- 29.

For gauge group $\bbU(1)$, a consistent and gauge invariant action on simplicial meshes can be obtained through a relatively simple construction and some applications can be completely analyzed \cite{ChrHal11SINUM}. A key ingredient is a notion of mass lumping \cite{HauLac93}.  Because of limitations in the scope of mass lumping, we don't expect this method to extend to the full Yang-Mills action.

By a more elaborate procedure, we define here, a consistent and gauge invariant discrete action for Yang-Mills theories on simplicial meshes. The simplexes are not congruent, so the metric must enter the formulas in a non trivial way, contrary to the cubical lattices of standard LGT. The metric on a simplex defines a mass matrix for the Whitney two-forms. Wilson loops associated with the faces of the simplexes are used to represent the curvature covariantly. Whereas standard LGT sums individual contributions from faces, we sum over simplexes, in which cross terms between the different faces appear. These couplings between Wilson loops are weighted by the mass matrix coefficients and made gauge invariant by discrete parallel transport between origins.

One advantage of the proposed  formulas is to allow local mesh refinements, useful for instance to efficiently represent singular fields. The consistency proof we provide covers such meshes, indicating robustness with respect to mesh geometry. Another advantage is to accommodate variable metrics as defined by Regge calculus \cite{Reg61}, see Remark \ref{rem:reg}.

This paper serves to introduce the formalism, give the definition, prove consistency and propose a discrete Noether's theorem. Preliminary applications to quantum field theory, including numerical results, have been reported elsewhere \cite{HalMac11}, while this paper was under review. It is organized as follows. Section \ref{sec:pre} contains definitions pertaining to connections and curvature, as well as Whitney forms. Section \ref{sec:def} includes the definition of the proposed discrete Yang-Mills Lagrangian and some comments. Section \ref{sec:con} contains the proof of consistency in a finite element sense. Finally section \ref{sec:noe} contains the discrete Noether's theorem we introduce.

\section{\label{sec:pre} Prerequisites}

\paragraph{Yang-Mills action} A standard reference for connections and curvature is \cite{KobNom63}. Here, we use notations as in \cite{ChrWin06}, which also contains a more comprehensive presentation of Lie algebra valued differential forms. The main ingredients are as follows.

Choose a compact Lie group $\bbG$ with associated Lie algebra $\frg$. For simplicity we suppose that $\bbG$ is a subgroup of the complex unitary $n \times n$ matrices, for some $n$. Typically an element of $\bbG$ will be denoted $G$ and an element of $\frg$ will be denoted $g$. The Hermitian conjugate of a matrix $g$ is denoted $g^\sch$ and the real scalar product is:
\begin{equation}\label{eq:scalp}
g \cdot g' = \rtr(g^\sch g').
\end{equation} 
When no confusion is possible with scalars, the unit matrix is denoted $\myone$ and the zero matrix is denoted $\myzero$.

Let $S$ be a bounded domain in $m$-dimensional Euclidean space (in applications $m=2,3,4$ would be the most common). The space of smooth $k$-forms on $S$ is denoted $\Omega^k(S)$.  The space $\Omega^k(S) \otimes \frg$ can be identified with the space of smooth $\frg$-valued $k$-forms on $S$. The bracket of Lie algebra valued forms is determined by:
\begin{equation}
[u \otimes g, u' \otimes g' ] =  (u \wedge u') \otimes [g,g'],
\end{equation}
where $u,u'$ are real valued differential forms and $g,g'$ are elements of $\frg$. 

A smooth connection one-form on $S$ is an element $A \in \Omega^1(S) \otimes \frg$ also called gauge field. Its curvature is $\calF(A) \in \Omega^2(S) \otimes \frg$ defined by:
\begin{equation}
\calF(A)= \rmd A + 1/2 [A,A].
\end{equation}
We will use such forms with less regularity, typically in some Sobolev space.

Gauge transformations of connection one-forms are associated with functions $Q: S \to \bbG$, and defined by:
\begin{equation}
\calG_Q(A)= QAQ^{-1} - (\rmD Q) Q^{-1}.
\end{equation}
One has:
\begin{equation}
\calF(\calG_Q(A)) = Q \calF(A) Q^{-1}.
\end{equation}

The Yang-Mills action is given by:
\begin{equation}\label{eq:ymaction}
\calS(A)=\int_S |\calF(A)|^2.
\end{equation}
Since the adjoint representation is unitary, this action is invariant under gauge transformations.

\paragraph{Whitney forms} We refer to \cite{Hip02}\cite{ArnFalWin06} for surveys on Whitney forms \cite{Wei52}\cite{Whi57} in a finite element guise. For aspects relating to differential geometry and algebraic topology, one can consult \cite{Pra07}. The following is a summary of results needed and serves mainly to introduce notations.

Let $\calT$ be a simplicial complex, spanning the domain $S$. The set of $k$-dimensional simplexes in $\calT$ is denoted $\calT^k$. Simplexes of dimension $0$, $1$ and $2$ are referred to as vertexes, edges, faces respectively. Generic labels for edges and faces will be $e$ and $f$ respectively. The symbol $T$ can be used for simplexes of any dimension. In the presence of several vertexes we denote them by $i,j,k,l$. We suppose an orientation has been chosen for each simplex in $\calT$. 

Let $W^k(\calT)$ be the space of Whitney $k$-forms on $\calT$. We also denote by $W^k(T)$ the space of Whitney $k$-forms on a simplex $T$. The canonical basis of $W^k(\calT)$ is denoted $(\lambda_{T})$, $T$ ranging over the set $\calT^k$ of
$k$-dimensional simplexes in $\calT$.  Explicitly, the $0$-forms are spanned by the barycentric coordinate maps. For any vertex $i \in \calT^0$, $\lambda_i$ is the piecewise affine map taking the value $1$ at vertex $i$ and $0$ at the other vertexes. For a  $k \geq 1$ and a $k$-dimensional simplex $T \in \calT$ with vertexes $i_0, \ldots, i_k$, ordered according to the chosen orientation of $T$, we have:
\begin{equation}
\lambda_{i_k \cdots i_0} = \lambda_T = k! \sum_{j = 0}^k (-1)^j \lambda_{i_j} \rmd \lambda_{i_0} \wedge \ldots  \widehat{\rmd \lambda_{i_j}} \ldots \wedge \rmd \lambda_{i_k}. 
\end{equation}
The hat signifies omission of this term. We will only use $0$-, $1$- and $2$-forms.  When $i,j$ are vertexes of an edge, the associated Whitney $1$-form is:
\begin{equation}
\lambda_{ji} = \lambda_i \rmd \lambda_j - \lambda_j \rmd \lambda_i.
\end{equation}
When $i,j,k$ are vertexes of a face, the associated Whitney $2$-form is:
\begin{equation}\label{eq:twobasis}
\lambda_{kji} = 2(\lambda_i \rmd \lambda_j \wedge \rmd \lambda_k - \lambda_j \rmd \lambda_i \wedge \rmd \lambda_k + \lambda_k \rmd \lambda_i \wedge \rmd \lambda_j).
\end{equation}

The space of (real) $k$-cochains consists of the functions that assign a real number to each $k$-simplex:
\begin{equation}
\calC^k(\calT)= \bbR^{\calT^k}.
\end{equation}
The coboundary operator $\delta : \calC^k(\calT) \to \calC^{k+1}(\calT)$ is defined as follows. The relative orientation $o(T,T')$ between a simplex $T \in \calT^{k+1}$ and a simplex $T' \in \calT^k$ is $0$ if $T'$ is not in the boundary of $T$, and $\pm 1$ when $T'$ is in the boundary of $T$, the sign depending on whether it is outward oriented or not. For $u \in \calC^k(\calT)$, $ \delta u \in \calC^{k+1}(\calT)$ is defined on any $T \in \calT^{k+1}$ by:  
\begin{equation}
(\delta u)_T = \sum_{T' \in \calT^k} o(T, T') u_{T'}.
\end{equation}
One has $\delta \delta = 0$. The de Rham map $R^k$ is defined by:
\begin{equation}
R^k : \mapping{\Omega^k(S)}{\calC^k(\calT)}{u}{(\int_T u )_{T \in \calT^k}} 
\end{equation}
In this formula the $k$-form $u$ is (pulled back and) integrated on the $k$-simplexes $T$, taking into account orientations. By Stokes' theorem we have commuting diagrams:
\begin{equation}\label{eq:drcom}
\xymatrix{
\Omega^k(S)\ar[r]^{\rmd}\ar[d]^{R^k} & \Omega^{k+1}(S) \ar[d]^{R^{k+1}}\\
\calC^k(\calT) \ar[r]^{\delta} & \calC^{k+1}(\calT)
}
\end{equation}
Since for $T,T' \in \calT^k$ one has:
\begin{equation}
\int_{T'} \lambda_T = \delta_{TT'} \rmins{(Kronecker delta),}
\end{equation}
the de Rham map induces isomorphisms:
\begin{equation}\label{eq:dri}
R^k : W^k(\calT) \to \calC^k(\calT), 
\end{equation}
whose inverses are:
\begin{equation}
\mapping{\calC^k(\calT)}{W^k(\calT)}{u}{\sum_{T \in \calT^k} u_T \lambda_T}.
\end{equation}
Given $u \in W^k(\calT)$ we denote by $u_\bs = R^k u$ its associated cochain. Another useful notation is the following. If $u \in \calC^k(\calT)$, one defines for any simplex $T\in \calT$ with vertexes $i_0,  \ldots, i_k$ :
\begin{equation}
u_{i_k \cdots i_0} = \pm u_T,
\end{equation}
the sign depending on whether the ordering of the vertexes agrees with the orientation of $T$ or not. For instance  we can write if $u \in \calC^2(\calT)$:
\begin{equation}
(\delta u)_{kji} = u_{ik} + u_{kj} + u_{ji}.
\end{equation}

Let $I^k$ denote the interpolation operator onto Whitney $k$-forms -- it is the projection onto $W^k(\calT)$ determined by the identity $R^kI^k = R^k$. Equivalently, for $u \in \Omega^k(S)$ one has:
\begin{equation}\label{eq:interpol}
I^ku = \sum_{T \in \calT^k} (\ts \int_{T} u) \lambda_{T} \in W^k(\calT).
\end{equation}
Interpolation commutes with the exterior derivative. 

Whitney forms are not smooth, but have enough regularity for the exterior derivative, in the sense of distributions/currents of Schwartz and de Rham, to be given by the simplex-wise definition (there are no Dirac measures on interfaces).

\section{\label{sec:def} Definition}

We first make some remarks on Lie algebra valued Whitney forms and their associated cochains. The de Rham isomorhisms (\ref{eq:dri}) give isomorphisms:
\begin{equation}
R^k : W^k(\calT)\otimes \frg \to \calC^k(\calT)\otimes \frg, 
\end{equation}
An assignment of a Lie algebra element to each $k$-simplex will be called a Lie algebra $k$-cochain. For $k= 1$ this goes as follows. Pick $A \in W^1(T)\otimes \frg$. Attached to an edge with vertexes
$i,j$ and oriented from $i$ to $j$, one has an element $A_{ji}\in
\frg$ obtained by writing:
\begin{equation}
A = \sum_{ji \in \calT^1} A_{ji} \lambda_{ji} \rmins{where} A_{ji} = \int_{ji} A,
\end{equation}
Thus $ji$ denotes an edge in $\calT^1$, oriented from the vertex $i$ to the vertex $j$. With these notations, the cochain associated with $A$ is:
\begin{equation}
A_\bs = (A_{ji})_{ji \in \calT^1} \in \calC^1(\calT)\otimes \frg.
 \end{equation}

The pullback of the $1$-form $A \in W^1(\calT) \otimes \frg$ to the edge $ji$ is a constant $1$-form. Therefore parallel transport from $i$ to $j$ is given simply by:
\begin{equation}\label{eq:ptua}
U_{ji}=\exp(-A_{ji}).
\end{equation}
We suppose $U_{ji}$ to be close enough to $\myone$ for the logarithm to be unambiguous. Then one is
free to think in terms of Lie group elements $U_{ji}$ (close to $\myone$) or Lie algebra elements $A_{ji}$ (close to $\myzero$). We use the conventions:
\begin{equation}\label{eq:usign}
U_{ij}^{\phantom{1}}=U_{ji}^{-1} \rmins{ and } U_{ii}= \myone,
\end{equation}
which correspond to:
\begin{equation}
A_{ij}=-A_{ji}\rmins{ and } A_{ii} = \myzero.
\end{equation}

A discrete gauge transformation is associated with a choice of Lie group elements
$G_i \in \bbG$, one for each vertex $i \in \calT^0$. One then transforms $A\in W^1(\calT) \otimes \frg$ by operating on its parallel transports. Namely, if $U_{ji}$ are the parallel transports of $A$, the parallel transports of its image $A'\in W^1(\calT) \otimes \frg$ will be $U'_{ji}$ defined by: 
\begin{equation}\label{eq:gt}
U'_{ji}=  G_j U_{ji} G_i^{-1}.
\end{equation}
With this choice of gauge transformations, we will construct
a gauge invariant approximation of the ``true'' Yang-Mills action on simplexes $T \in \calT^m$ of maximal dimension, which we recall to be defined by:
\begin{equation}\label{eq:action}
\calS_T(A)=\int_T |\calF(A)|^2.
\end{equation}
In our setting $T$ inherits the Euclidean metric of the ambient space, but as already indicated one could use a Regge metric instead. The metric enables integration of scalar functions on
$T$. It also gives, at each point $x$ of $T$, a scalar product on alternating forms above $x$. The associated norm was denoted $| \cdot |$ in (\ref{eq:action}).

Let $M$ be the matrix of the $\rmL^2(T)$ product on $W^2(T)$, equipped with the standard
basis defined by (\ref{eq:twobasis}). This matrix is indexed by the two-dimensional faces of $T$ (and depends on their orientations). Explicitly, for two faces $f_0$ and $f_1$  in a simplex $T$, we put:
\begin{equation}\label{eq:massmat}
M_{f_0f_1}(T) = \int_T \lambda_{f_0} \cdot \lambda_{f_1}.
\end{equation}
where the scalar product of alternating forms is denoted $(\cdot)$. Notice that this matrix is not diagonal, which can be interpreted as an interaction between neighboring faces.

Given parallel transports $U \in \calC^1(\calT, \bbG)$, the discrete curvature associated with a face $f \in \calT^2$ with vertexes $i,j,k$ is defined in analogy with square Wilson loops \cite{Wil74} by:
\begin{equation}\label{eq:holdef}
F_{kji}=U_{ik}U_{kj}U_{ji}.
\end{equation}
In other words, one considers the holonomy around the boundary of the face $f$ of the $1$-form $A \in W^1(\calT) \otimes \frg$ related to $U$ by (\ref{eq:ptua}). This formula depends on the ordering of the vertexes and locates the curvature at the vertex $i$. The curvature at vertex $j$ is obtained by permuting indices and satisfies:
\begin{equation}
F_{ikj}=U_{ji}F_{kji}U_{ij}.
\end{equation}
This gives a formula for parallel transport of curvature from $i$ to $j$. Concerning orientation of a given face, we also notice that the definition (\ref{eq:holdef}) implies:
\begin{equation}
F_{jki}=F_{kji}^{-1}.
\end{equation}
Under gauge transformations this curvature behaves as follows. If $F \in \calC^2(\calT, \bbG)$ is the curvature associated with holonomies $U\in \calC^1(\calT, \bbG)$ by (\ref{eq:holdef}), and $U$ is transformed by $G \in \calC^0(\calT, \bbG)$ to $U'$, according to (\ref{eq:gt}), then the curvatures $F'$ of $U'$ are:
\begin{equation}\label{eq:fgt}
F_{kji}' = G_i F_{kji} G_i^{-1}.
\end{equation}
When $f$ is a face with vertexes $i,j,k$, which is ordered as $i \to j \to k$ and we choose to locate the curvature at $i$, we put:
\begin{equation}
F_{f} = F_{kji}.
\end{equation}
This formula defines the curvature of a \emph{pointed oriented face} $f$. For a pointed face $f$, its distinguished point is denoted $\dot{f}$ and called its origin.

We now propose the following definition of a discrete action for lattice gauge theory on simplexes, as an alternative to (\ref{eq:action}):
\begin{definition} We define:
\begin{align}\label{eq:sgtdef}
\calS'_T(A) & = \sum_{f_0 f_1} M_{f_0 f_1}(T) \rtr \big(U_{\dot{f}_1\dot{f}_0}(\myone-F_{f_0}^\sch) U_{\dot{f}_0\dot{f}_1} (\myone-F_{f_1})\big).
\end{align}
In this formula we sum over  pairs of faces $f_0, f_1$ of $T$, each one having an orientation and a distinguished point. We have incorporated the parallel transport determined by $A$, between the distinguished points $\dot{f}_0$ and $\dot{f}_1$. 
\end{definition}

We can state:
\begin{theorem}
The action $\calS'_T$ is  discretely gauge invariant, with respect to transformations of the form (\ref{eq:gt}).
\end{theorem}
Global actions are obtained by summing the contributions of each maximal simplex in $\calT^m$, so that for $A \in W^1(\calT) \otimes \frg$:
\begin{equation}
\calS'(A) = \sum_{T\in \calT^m} \calS'_T(A).
\end{equation}

\begin{remark} Define, for any $U\in \calC^1(\calT, \bbG)$, subject to (\ref{eq:usign}) and any $F \in \calC^2(\calT, \bbG)$ (not necessarily the curvature of $U$!):
\begin{equation}
 \calL_T(U,F)  = \sum_{f_0 f_1} M_{f_0 f_1}(T) \rtr \big(U_{\dot{f}_1\dot{f}_0}(1-F_{f_0}^\sch) U_{\dot{f}_0\dot{f}_1} (1-F_{f_1})\big).
\end{equation}
For fixed $U$,  $\calL_T(U,F)$ can be interpreted as an $\rmL^2(T)$ norm squared of $F$. This norm depends on $U$, which contrasts with the fact that the usual  $\rmL^2(T)$ norm, as it appears in (\ref{eq:action}) is independent of the gauge.

However  $\calL_T(U,F)$ is invariant under transformations $(U,F) \mapsto (U',F')$, associated with some $G \in \calC^0(\calT, \bbG)$ by (\ref{eq:gt}) and (\ref{eq:fgt}).
%, similarly to the fact the usual  $\rmL^2(T)$ norm is invariant under the adjoint representation.
\end{remark}

\begin{remark} It is most natural to compute a norm in the Lie algebra. Thus in the definition of the discrete actions, the terms of the form $(F_f -\myone)$ should be considered as approximations of $\log F_f$ that are more readily computable. For a general Lie group one could use:
\begin{equation}
\calS'_T(A) = \sum_{f_0 f_1} M_{f_0 f_1}(T) \Ad(U_{\dot{f}_1\dot{f}_0})\log(F_{f_0}) \cdot \log(F_{f_1}).
\end{equation}
Here $\Ad: \bbG \to \mathrm{End}(\frg)$ is the adjoint representation. Recall that for a given $U \in \bbG$, $\Ad(U): \frg \to \frg$ is the tangent map at unity, of the automorphism of $\bbG$ mapping an element $G$ to $U G U^{-1}$. The scalar product on $\frg$, denoted here with $(\cdot)$, should make the adjoint representation unitary.
\end{remark}

\begin{remark}\label{rem:reg}
The proposed method can be combined seamlessly with Regge calculus \cite{Reg61} (see \cite{Chr04M3AS}\cite{Chr11NM} for finite element interpretations). In Regge calculus the metric of a given simplex is determined by the edge lengths and yields a mass matrix, as in the adopted setting. It seems that only minor modifications are necessary for the consistency proof we will provide, to cover the case where the local metrics are Regge metrics interpolating a smooth one. Notice that this action uses just edge lengths  for the metric and values of the fields at vertexes or edges, never vertex coordinates, so that the method is "coordinate free".
 \end{remark}

\paragraph{Scalar fields.} We include a definition of a discrete action for certain so-called scalar fields. We will not prove consistency for it here.

Let $V$ be an inner product space on which $\bbG$ acts unitarily. The action is denoted simply $(G,v) \mapsto G \Gact v$. Likewise the associated action of $\frg$ on $V$ is denoted $(g,v) \mapsto g \gact v$.

Let $\nabla$ denote the canonical flat connection acting on sections $\Phi: S \to V$. Given $A$, the action to approximate is: 
\begin{equation}
\calS_T(A,\Phi) = \int_T | \nabla \Phi + A \gact \Phi|^2.
\end{equation}

Let then $\Phi \in W^0(\calT) \otimes V$ be a discrete scalar field. We can write:
\begin{equation}
\Phi = \sum_{i \in \calT^0} \Phi_i \lambda_i \rmins{with} \Phi_i = \Phi(i).
\end{equation}
Concerning the cochains associated with $\Phi$ and $\nabla \Phi \in W^1(\calT) \otimes V$ we have, by (\ref{eq:drcom}):
\begin{equation}
\nabla \Phi = \sum_{ji \in \calT^1} (\Phi_j - \Phi_i) \lambda_{ji}. % \rmins{with} (\delta \Phi)_{ji} = \Phi_j -\Phi_i.
\end{equation}
The mass matrix for Whitney $1$-forms is also denoted by $M$ and is indexed by oriented edges. Thus:
\begin{equation}
M_{e_0 e_1}(T) = \int_T \lambda_{e_0} \cdot \lambda_{e_1}.
\end{equation}
For an oriented edge $e$ we denote its origin by $\dot{e}$ and its target by $\ddot{e}$. %The edge oriented the other way is denoted $\bar{e}$.

As a discrete action we propose to use:
\begin{equation}
\calS'_T(A,\Phi) = \sum_{e_0 e_1} M_{e_0e_1}(T) U_{\ddot{e}_1\ddot{e}_0} \Gact (\Phi_{\ddot{e}_0} - U_{{e}_0}\Gact \Phi_{\dot{e}_0}) \cdot (\Phi_{\ddot{e}_1} - U_{{e}_1}\Gact\Phi_{\dot{e}_1}\big),
\end{equation}
where the scalar product is that of $V$. 

Recall that under discrete gauge transformations associated with $G_i \in \bbG$, the parallel transports $U$ transform by (\ref{eq:gt}). The corresponding transformation of $\Phi$ is, at the level of cochains:
\begin{equation}
\Phi_i \mapsto G_i \Gact \Phi_i.
\end{equation}
It is readily checked that $\calS'(A,\Phi)$ is discretely gauge invariant.

\paragraph{Conventional LGT} For comparison, we recall the usual definition of LGT on cubical meshes. One attaches a discrete parallel transport $U_{ji}$ to any two vertexes $i,j$ of the grid linked by an edge, with the preceding constraint $U_{ji} = U_{ij}^{-1}$. A face $f$ of this mesh is then a square with four vertexes, called plaquette. Given a choice of orientation and origin, these four vertexes can be labelled $f_0,f_1,f_2,f_3$. Barring the coupling constant, the action is then defined by a sum over faces:
\begin{equation}\label{eq:standardlgt}
\sum_{f} \rtr (\myone- U_{f_0f_3} U_{f_3f_2} U_{f_2f_1} U_{f_1f_0}),
\end{equation}
Remark that the action is independent of the choice of origin and orientation of the faces. 

Thus in standard LGT, one sums over faces, and the contribution of each face is discretely gauge invariant, under transformations (\ref{eq:gt}). On the other hand, for simplicial meshes, we propose to sum over maximal simplexes (tetrahedrons in dimension 3), with gauge invariant cross terms between neighboring faces. The counterpart for cubical meshes would be to sum over cubes, inside which plaquettes are coupled two by two. From this point of view, standard LGT uses just diagonal terms. This is also why the scalar product (\ref{eq:scalp}) is more apparent in (\ref{eq:sgtdef}) than in (\ref{eq:standardlgt}). One can compare with the fact that the Yee scheme \cite{Yee66} can be deduced from a finite element scheme via mass lumping \cite{Mon93} (whereby mass matrices are approximated by diagonal matrices in a consistent way).  

That the continuum limit of the LGT action is the Yang Mills action, is usually argued on the basis of Taylor expansions and a couple of terms in the BCH formula, e.g. \cite{Kog83} p. 786.  To a numerical analyst these arguments would prove consistency in the finite difference sense.  Consistency in the finite element sense is related of course, but not identical, putting emphasis on the choice of function theoretic norms, typically Sobolev norms. Transposing the finite element arguments we will give here for the simplicial case, to the cubical case (with tensor-product Whitney forms) would give a novel proof on the convergence of the Wilson action to the Yang-Mills action, in the continuum limit. We don't expect  consistency proofs based on Taylor expansion techniques to carry over directly to the simplicial setting, since it is difficult for them to take into account mesh geometry.

\section{\label{sec:con} Consistency}

In this section we want to study the error committed, when approximating (\ref{eq:action}) by (\ref{eq:sgtdef}). For this purpose we introduce two more discrete actions defined for $A \in W^1(T) \otimes \frg$, for simplexes $T \in \calT^m$ . These discrete actions serve only to provide intermediate steps between $\calS_T(A)$ and $\calS'_T(A)$, aiming at clarifying the analysis. 

First we define:
\begin{align}
\calS^1_T(A) & = \int_T | I^2 \calF(A)|^2,
\end{align}
and remark that, by (\ref{eq:scalp}), (\ref{eq:interpol}) and (\ref{eq:massmat}):
\begin{align}
  \calS^1_T(A) & = \sum_{f_0 f_1} M_{f_0 f_1}(T) \rtr \big(\ts\int_{f_0} \calF(A)^\sch \ts\int_{f_1} \calF(A)\big).
\end{align}
The sum extends over pairs of faces $f_0, f_1$ of $T$. Second, given also a choice of origins of the faces, we define:
\begin{align}
\calS^2_T(A) & = \sum_{f_0 f_1} M_{f_0 f_1}(T) \rtr \big((\myone- F_{f_0})^\sch (\myone-F_{f_1})\big).
\end{align}

We now evaluate the error in each of the approximations:
\begin{equation}
\calS_T \to \calS^1_T\to \calS^2_T \to \calS'_T.
\end{equation}
The reasoning is partly similar to \cite{ChrHal11SINUM}, but we need to extend to non-commutative gauge group and abandon the convenience of mass-lumping. For definiteness we restrict attention to dimension $m=3$ for the ambient space $S$. Maximal simplexes are then tetrahedrons.

We suppose that we have a regular sequence of simplicial meshes $\calT_n$ of the domain $S$. The diameter of a simplex $T$ is denoted $h_T$, and  the biggest $h_T$ when $T$ is in $\calT_n$ is denoted $h_n$. We suppose that the sequence $(h_n)_{n \in \bbN}$ converges to $0$. Let $I^k_n$ denote the interpolant onto Whitney $k$-forms associated with the mesh $\calT_n$. Let $X_n$ denote the space  $W^1(\calT_n) \otimes \frg$. For ease of notation put also $I_n = I^1_n$ and $J_n= I^2_n$.

The following definition is the natural extension of \cite{Cia91} \S 28 to a non-linear setting.
\begin{definition}
Let $\| \cdot\|$ denote the norm of some Banach space in which $\Omega^1(S) \otimes \frg$ is dense, and containing each $X_n$. We say that two actions $\calS_n$ and  $\calS'_n$ defined on $X_n$ are consistent with each other, with respect to  $\| \cdot \|$, if for all smooth $A$ we have:
\begin{equation}
\sup_{A' \in X_n} |\rmD \calS_n(I_n A) A' - \rmD \calS'_n(I_nA)A'| / \| A'\| \to 0 \rmins{when} n \to \infty.
\end{equation}
More precisely if the above expression is $\calO(\epsilon_n)$, for some sequence $\epsilon = (\epsilon_n) \to 0$, we speak of consistency of order $\epsilon$.
\end{definition}

If there is a constant $C>0$ (which may depend on $A$ and the sequence $(\calT_n)$ but not on $n$) such that quantities $a_n$ and $b_n$ satisfy $a_n \leq C b_n$ for all $n$, we write $a_n \cleq b_n$ or $a_n = \calO(b_n)$.

Consider a simplex $T$ of dimension $d$ and let $\Phi : \hat T \to T$ be a scaling map of the form $\Phi(x) = h_T x + y$. We have, for any  $u \in \Omega^k(T)$:
\begin{equation}\label{eq:scaling}
\| u \|_{\rmL^p(T)} = h_T^{-k + d/p} \| \Phi^\star u \|_{\rmL^p(\hat T)}.
\end{equation}
Arguments based on  this identity will be referred to as \emph{scaling} arguments. 

For instance if we have a sequence of elements $u^n \in X_n$ which is bounded in $\rmL^4(S)$ and $e_n$ is an edge in $\calT_n$ we deduce by scaling that we have bounds:
\begin{equation}
|\int_{e_n} u^n| \cleq h_{e_n}^{1/4}.
\end{equation}
which gives:
\begin{equation}
|u^n_\bs|_{\ell^\infty} = \max_{e \in \calT_n^1} |u^n_e| \cleq h_n^{1/4}.
\end{equation} 
This is enough to guarantee that the logarithm is unambiguous as required initially.

On cochains, we  consider norms of the following form (with no coefficients):
\begin{equation}
\rmins{for } u \in \calC^k(\calT) \quad |u|^2 = \sum_{T \in \calT^k} |u_T|^2.
\end{equation}
We will frequently use that on reference simplexes $\hat T$, the cochain norm is equivalent to functional norms, as they appear in the right hand side of (\ref{eq:scaling}).

We let $A \in \Omega^1(S) \otimes \frg$ be smooth. We put $A^n = I_n A$. Remark that for edges $e$ in $\calT_n$ we have:
\begin{equation}
A^n_e= (I_nA)_e= \int_e A = \calO (h_e),
\end{equation}
and for faces $f$ in $\calT_n$ we have:
\begin{equation}
(\delta A^n_\bs)_f = (\rmd I_n A)_f=  \int_f \rmd A = \calO (h_f^2).
\end{equation}

\paragraph{Step one}
We compare $\calS$ and $\calS^1$. We first remark:
\begin{lemma} We have:
\begin{equation}
\|\calF(A^n) - J_n \calF(A^n)\|_{\rmL^2(T)} = \calO(h_T\|\calF(A^n)\|_{\rmL^2(T)}),
\end{equation}
and:
\begin{equation}
\|J_n \calF(A^n)\|_{\rmL^2(T)} \cleq \|\calF(A^n)\|_{\rmL^2(T)}.
\end{equation}
\end{lemma}
\begin{proof}
By scaling, knowing that $\calF(A^n)$ lives in the space of Whitney forms of maximal polynomial order $2$ (\cite{ChrWin06} \S 3.3).
\end{proof}

Recall the formula:
\begin{equation}
\rmD|_{A} \calF(A) A' = \rmd A' + [A, A'].
\end{equation}
Since Whitney forms are stable under the exterior derivative, we have:
\begin{equation}
\rmD \calF(A^n) A' - J_n \rmD \calF(A^n) A' = [A^n, A'] - J_n [A^n, A'] .
\end{equation}

\begin{lemma}We have:
\begin{equation}
\| [A^n, A'] - J_n [A^n,A'] \|_{\rmL^2(T)} = \calO( h_T \|A'\|_{\rmL^2(T)}).
\end{equation}
\end{lemma}
\begin{proof}
Remark that the interpolation is exact on a tetrahedron if $A^n$ is constant on it. On a reference tetrahedron we can therefore write:
\begin{equation}
\|[A^n, A'] - J_n [A^n,A'] \|_{\rmL^2(\hat{T})} \cleq \| \nabla A^n\|_{\rmL^\infty(\hat{T})} \| A'\|_{\rmL^2(\hat{T})}.
\end{equation}
The estimate  on $T$ then follows by scaling.
\end{proof}

\begin{proposition}\label{prop:stepone}
The actions $\calS$ and $\calS^1$ are consistent of order $h$ for the $\rmL^2$ norm.
\end{proposition}
\begin{proof}
We have:
\begin{equation}
\rmD \calS_T (A^n) A' = \int_T \calF(A) \cdot \rmD\calF(A) A',
\end{equation}
and:
\begin{equation}
\rmD \calS^1_T (A^n) A' = \int_T J_n\calF(A) \cdot \rmD J_n \calF(A) A'.
\end{equation}
With a more compact notation we can evaluate the difference:
\begin{align}
%  & \ |\ts \int \calF \cdot \rmD \calF A' - \int J \calF \cdot J \rmD \calF A'|\\
 |\ts \int (\calF - J \calF) \cdot \rmD \calF A' + \int J \calF \cdot (\rmD \calF - J \rmD \calF) A'|
\cleq   h_T \|\calF\|_{\rmL^2(T)} \| A'\|_{\rmL^2(T)}.
\end{align}
Summing these estimates for all the tetrahedrons and applying the Cauchy-Schwarz inequality gives the result.
\end{proof}

\paragraph{Step two}
We compare now $\calS^1$ and $\calS^2$.

\begin{lemma}\label{lem:fexpl} Let a face $f$ have vertexes $i,j,k$. We have:
\begin{align}
\int_f \calF(A^n) &=   \int_f \rmd A^n + \frac{1}{2} [A^n, A^n], \label{eq:intff}\\
&=  A_{ik}^n + A_{kj}^n + A_{ji}^n  + \frac{1}{6}\big([A^n_{ji}, A^n_{kj}] + [A^n_{kj}, A^n_{ik}] + [A^n_{ik}, A^n_{ji}] \big).
\end{align}
\end{lemma}
\begin{proof}
Remark that:
\begin{equation}
\int_{f} \rmd A^n = A_{ik}^n + A_{kj}^n + A_{ji}^n,
\end{equation}
Using formulas of the type:
\begin{equation}
\int_{f} \lambda_{ji} \wedge \lambda_{kj} = 1/6,
\end{equation}
one gets the second term.
\end{proof}

\begin{proposition}\label{prop:ferror} We have:
\begin{equation}
\myone- F_f(A^n) = \int_f \calF(A^n) + \calO(h_f^3).
\end{equation}
\end{proposition}
\begin{proof}
$F_f(A^n)$ can be estimated with the help of the BCH formula and compared with the formula previously obtained for the right hand side.
% See also \cite{Tay96II} \S C.5.
\end{proof}

Using the same arguments as in the proof of Lemma \ref{lem:fexpl}, we get:
\begin{lemma} We have:
\begin{align}
\rmD|_A (\ts \int_{f} \calF(A)) A' = \  & \ts \int_{f} \rmd A' + [A, A']\\
 = \ & A'_{ik} + 1/6 ([A^n_{ji}, A'_{ik}] - [A^n_{kj}, A_{ik}']) + \\
     & A'_{kj} + 1/6 ([A^n_{ik}, A_{kj}'] - [A^n_{ji}, A_{kj}']) + \\
     & A'_{ji} + 1/6 ([A^n_{kj}, A'_{ji}] - [A^n_{ik}, A'_{ji}]). 
\end{align}
\end{lemma}

\begin{proposition}\label{prop:dferror}
We have:
\begin{equation}
-\rmD|_A F_{f}(A)A' = \rmD|_A \ts \int_{f} \calF(A) A' + \calO( h_f |\delta A'_\bs| + h_f^2|A'_\bs| ). 
\end{equation}
\end{proposition}
\begin{proof}
Define an entire function $\phi$, by setting, for $z\neq 0$:
\begin{equation}
\phi(z)= (1 -e^{-z})/z.
\end{equation}

Recall that:
\begin{equation}
F_{kji}= \exp(-A_{ik}) \exp(-A_{kj}) \exp(-A_{ji}),
\end{equation}
so that:
\begin{align}
-\rmD F_{kji}(A)A' =\ & \exp(-A_{ik}) \phi(\ad(-A_{ik}))A'_{ik}\exp(-A_{kj}) \exp(-A_{ji}) +\\
& \exp(-A_{ik}) \exp(-A_{kj}) \phi(\ad(-A_{kj}))A'_{kj} \exp(-A_{ji}) +\\
& \exp(-A_{ik}) \exp(-A_{kj}) \exp(-A_{ji}) \phi(\ad(-A_{ji}))A'_{ji}.
\end{align}
Expand using $\phi'(0) = -1/2$ and rearrange to obtain, up to the announced error term, the previously computed right hand side.
\end{proof}

\begin{proposition}\label{prop:steptwo}
The discrete actions $\calS^1$ and $\calS^2$ are consistent of order $h$ with respect to the norm defined by:
\begin{equation}
\|A\| = \|\rmd A\|_{\rmL^2} + \| A \|_{\rmL^2}.
\end{equation}
\end{proposition}
\begin{proof}
Let $f_0$ and $f_1$ be faces of a tetrahedron $T$ in $\calT_n$.
Define:
\begin{equation}
L_n A' = \rmD |_{A= A^n} \big (\ts\int_{f_0} \calF(A)^\sch \ts\int_{f_1} \calF(A) \big ) A'.
\end{equation}
and:
\begin{equation}
L'_n A' =  \rmD|_{A= A^n} \big ((\myone-F_{f_0}(A))^\sch (\myone-F_{f_1}(A))\big ) A'.
\end{equation}
Combining Propositions \ref{prop:ferror} and \ref{prop:dferror} gives:
\begin{equation}\label{eq:lml}
L_n' A' =  L_nA' + \calO (h_T^3 |\delta A'_\bs| + h_T^4|A'_\bs|). 
\end{equation}

We have, by scaling:
\begin{equation}
|M_{f_0 f_1}| \cleq h_T^{-1}.
\end{equation}
Insert it in the error term of ($\ref{eq:lml}$) and write:
\begin{align}
h_T^2 |\delta A'_\bs|_T + h_T^3 |A'_\bs|_T \cleq  h_T^{2+1/2} \| \rmd A'\|_{\rmL^2(T)} + h_T^{3  - 1/2} \|A'\|_{\rmL^2(T)}.
\end{align}

We sum over all tetrahedrons and apply a Cauchy-Schwarz inequality, remarking that:
\begin{equation}
(\sum_{T \in \calT^3} h_T^5)^{1/2} \cleq h_n.
\end{equation}
This concludes the proof.
\end{proof}

\paragraph{Step three}
We compare $\calS^2$ and $\calS'$.

\begin{proposition}We have:
\begin{equation}
U_{li}F_{kji}U_{il} = F_{kji} + \calO(h_T^3),
\end{equation}
and also:
\begin{equation}
\rmD|_A U_{li}(A)F_{kji}(A)U_{il}(A) A' = \rmD|_A F_{kji}(A) A' + \calO(h_T |\delta A'_\bs| + h_T^2 |A'_\bs|).
\end{equation}
\end{proposition}
\begin{proof}
For the first assertion we write:
\begin{equation}
U_{li}F_{kji}U_{il} - F_{kji} = U_{li}(F_{kji} - \myone) U_{il} - (F_{kji} -\myone),
\end{equation}
and conclude using:
\begin{equation}\label{eq:fm1}
F_{kji} - \myone  =  \calO(h_T^2).
\end{equation} 

For the second one we compute:
\begin{align}
\rmD|_A U_{li}F_{kji}U_{il} A'  = U_{li} \rmD|_A F_{kji}(A) A' U_{il} - U_{li} [\phi(\ad(-A_{li})) A_{li}', F_{kji}] U_{il}.
\end{align}
On the right hand side, in the second term, we can replace $F_{kji}$ by $F_{kji} - \myone$ and use again (\ref{eq:fm1}). From this the second assertion follows.
\end{proof}

As in the previous paragraph we can deduce:
\begin{proposition}\label{prop:laststep}
The discrete actions $\calS^2$ and $\calS'$ are consistent of order $h$ with respect to the norm defined by:
\begin{equation}
\|A\| = \|\rmd A\|_{\rmL^2} + \| A \|_{\rmL^2}.
\end{equation}
\end{proposition}

\paragraph{Conclusion}
Adding the three estimates proved in Propositions \ref{prop:stepone}, \ref{prop:steptwo} and \ref{prop:laststep}, we get:
\begin{theorem}
The discrete actions $\calS$ and $\calS'$ are consistent of order $h$ with respect to the norm defined by:
\begin{equation}
\|A\| = \|\rmd A\|_{\rmL^2} + \| A \|_{\rmL^2}.
\end{equation}
\end{theorem}

The arguments introduced also immediately show that, concerning the action itself, we have consistency of order $h^2$. That is, if $A$ is a smooth gauge potential, we have:
\begin{equation}
\calS_n(I_nA) - \calS'_n(I_nA) = \calO(h_n^2).
\end{equation}

\section{\label{sec:noe} A discrete Noether's theorem}
In this section we propose an analogue of Noether's first theorem \cite{Olv93}, expressed for discretizations over simplicial complexes, when the group acting on the fields preserves fibers, as defined below. This discrete result provides a discrete conservation law associated with discrete gauge invariance. While it does not capture the full power of the continuous one, it is sufficient to prove constraint preservation for evolution problems as in \cite{ChrHal11IMA}. Recall that non-invariance of Lie algebra valued Whitney forms under discrete gauge transformations makes the standard action (\ref{eq:action}) problematic for the simulation of evolution, because constraints are not preserved \cite{ChrWin06}.  For electromagnetics, the conservation law associated in the continuum with gauge invariance, is nothing but electric charge conservation, of great physical significance.

Related discrete Noether's theorems have been discussed in particular in \cite{ManQui05}\cite{Man06}.

We suppose we have a finite simplicial complex $\calT$. The maximal dimension of the simplexes in $\calT$ is $m$.  We write $S \subsimp T$ to say that $S$ is a subsimplex of $T$. If $S$ is a simplex and $i$ a vertex not in $S$, $S+i$ is the simplex obtained by adjoining the vertex $i$ to $S$. Conversely, if $S$ is a simplex and $i$ a vertex of $S$, $S-i$ is the face of $S$ opposite $i$.

We suppose that on each maximal simplex $T\in \calT^m$ we have attached fields $\Phi_T$ of the form:
\begin{equation}
\Phi_T = (\Phi_T(S))_{S \subsimp T} \in \prod_{S \subsimp T} V_T(S).
\end{equation}
That is, $\Phi_T$ attaches a value in some space $V_T(S)$ to each subsimplex $S$ of $T$. We call $V_T(S)$ the fiber above $S$.

We suppose that we have a Lagrangian $\calL_T$ attached to $T$, which is a function:
\begin{equation}
\calL_T : \prod_{S \subsimp T} V_T(S) \to \bbR.
\end{equation}

In the following we fix a simplex $T \in \calT^m$.
We suppose we have a one parameter group action $\Lambda_T$ which acts separately on each fiber $V_{T}(S)$:
\begin{equation}
\Lambda_{T}(S): \bbR \to \mathrm{Aut}( V_T(S)),
\end{equation}
and for $t \in \bbR$: 
\begin{equation}
\Lambda_{T}[t] \Phi_T = (\Lambda_{T}(S)[t]\Phi_T(S))_{S \subsimp T}.
\end{equation}
We suppose that this group action leaves $\calL_T$ invariant:
\begin{equation}
\forall t \in \bbR \quad \calL_T(\Lambda_T[t]\Phi_T) = \calL_T(\Phi_T).
\end{equation}
We define the (local) infinitesimal generators:
\begin{equation}
\xi_T(S)  =  \partial|_{t=0} \Lambda_T(S)[t]\Phi_T(S).
\end{equation}
and the (local) Euler-Lagrange functions:
\begin{equation}
E_T(S) = \partial|_S \calL_T(\Phi_T),
\end{equation}
and put:
\begin{equation}
F_T(S) = E_T(S) \xi_T(S).
\end{equation}

For each simplex $S\subsimp T$ and each $i \in T \setminus S$ choose a number $p_T(i,S)$ subject to the condition that, for any simplex $S'$ of dimension at least $1$:
\begin{equation}\label{eq:discpart}
\sum_{i \in S'} p_T(i, S'-i) = 1.
\end{equation}

\begin{proposition}\label{prop:discnoether}
Define, for any vertex $i\in T$:
\begin{equation}
W_T(i) = F_T(i) + \sum_{S\subsimp T: i \not \in S} p_T(i,S) F_T(S+i),
\end{equation}
and for any two distinct vertexes $i,j\in T$: 
\begin{equation}
V_T(i,j) = F_T(i) - F_T(j) + \sum_{S\subsimp T: i,j \not \in S} p_T(i,S) F_T(S+i) - p_T(j,S)F_T(S+j).
\end{equation}
Then we have:
\begin{equation}\label{eq:noether}
(m+1)W_T(i) = \sum_{j:j \neq i} V_T(i,j).
\end{equation}
\end{proposition}
\begin{proof}
In this proof, in which $T$ is fixed, we drop the index $T$. The summation variables $S,S'$ are subsimplexes of $T$. First we remark:
\begin{eqnarray}
&  & \sum_{j: j \neq i} \big( F(i) + \sum_{S: i,j \not \in S} p(i,S) F(S+i) \big)\\
& = & m W(i) - \sum_{j: j \neq i} \sum_{S: \substack{i \not \in S\\ j \in S}} p(i,S) F(S+i).
\end{eqnarray}
then we remark: 
\begin{eqnarray}
& &  \sum_{j: j \neq i} \big( F(j)  + \sum_{S: i,j \not \in S} p(j,S)F(S+j) + \sum_{S: \substack{i \not \in S\\ j \in S}} p(i,S) F(S+i)\big)\\
& = &  \sum_{j:j \neq i} \big(F(j) + \sum_{S': j \in S'} p(j,S'-j)F(S') \big),\\
& = &  \sum_{S'} F(S')  -  W(i).
\end{eqnarray}
From invariance of the Lagrangian we get :
\begin{equation}
\sum_{S'} F(S') = 0,
\end{equation}
and this concludes the proof.
\end{proof}

In the applications we have in mind, if a simplex $S\in \calT$ is included in two maximal simplexes $T,T' \in \calT^m$ we have   $V_T(S)= V_{T'}(S)$, and the global variable $\Phi$ has the property $\Phi_T(S)= \Phi_{T'}(S) $. When this happens for all choices $S,T, T'$ such that $S \subsimp T, T' \in \calT^m$,  we have a well defined fiber $V_S$ above each $S\in \calT$ and the action $\calS$ will be of the form:
\begin{equation}
\calS = \sum_{T \in \calT^m} \calL_T : \prod_{S \in \calT} V(S) \to \bbR.
\end{equation}
Moreover we suppose that the group action $\Lambda$ acts separately on the fibers $V(S)$, independently of any embedding into a maximal simplex $T$. In this setting we define the (global) infinitesimal generators:
\begin{equation}
\xi(S) = \partial|_{t=0} \Lambda(S)[t]\Phi(S), 
\end{equation}
 the (global) Euler Lagrange functions:
\begin{equation}
E(S) = \partial|_S \calL(\Phi) = \sum_{T\in \calT^m: S \subsimp T} E_T(S),
\end{equation}
and put:
\begin{equation}
F(S) = E(S) \xi(S).
\end{equation}
We suppose finally that we have chosen the numbers $p_T(i,S)$ independently of $T$ containing $i$ and $S$. When $i$ and $S$ are not included in any simplex of $\calT$ we set $p(i,S)=0$. The preceding Proposition gives, by adding contributions from all maximal simplexes $T$:

\begin{proposition}\label{prop:discnoetherglob}
Define, for any vertex $i\in \calT$:
\begin{equation}
W(i) = F(i) + \sum_{S\in \calT : \substack{S+i \in \calT\\i \not \in S}} p(i,S) F(S+i),
\end{equation}
and for any two distinct vertexes $i,j\in \calT$ linked by an edge: 
\begin{equation}
V(i,j) = \sum_{T\in \calT^m:i,j \in T} V_T(i,j).
\end{equation}
Then we have:
\begin{equation}\label{eq:noetherglob}
(m+1)W(i) = \sum_{j:i+j\in \calT^1} V(i,j).
\end{equation}
\end{proposition}

In brief, equation (\ref{eq:noetherglob}) expresses a weighted sum of (global)  Euler-Lagrange functions applied to infinitesimal generators, as a discrete divergence. Indeed it is natural to think of $V(i,j)$ as degrees of freedom of a vectorfield $V$. Choose any cellular complex dual to $\calT$, so that, in particular, the domain is covered by cells dual to the vertexes $i \in \calT^0$. Then $V(i,j)$ is the flux from the cell dual to $i$ into the cell dual to $j$, through the dual face of the edge $ij \in \calT^1$. The right hand side of $(\ref{eq:noetherglob})$ is then the total flux leaving the cell dual to $i$, which is the natural degree of freedom for the divergence of $V$. The essential antisymmetry property  $V(i,j) = -V(j,i)$ guarantees that summing the discrete divergence over a union of top-dimensional dual cells, leaves only a boundary term.

These considerations apply directly to the proposed simplicial gauge theory, for which moreover we have variables attached only to $0$- and $1$- simplexes.

\section*{Acknowledgments}
We are grateful to Elizabeth Mansfield for helpful comments on Noether's theorems.

This work, conducted as part of the award ``Numerical analysis and simulations of geometric wave equations''  made under the European Heads of Research Councils and European Science Foundation EURYI (European Young Investigator) Awards scheme, was supported by funds from the 
Participating Organizations of EURYI and the EC Sixth Framework Program.

\bibliography{alexandria}{}
\bibliographystyle{plain}

\end{document}